\documentclass[a4paper,11pt]{amsart}
\usepackage{amssymb,amsmath,amsfonts,amsthm}
\usepackage[english]{babel}
\usepackage[yyyymmdd]{datetime}
\usepackage{enumerate,expdlist}
\usepackage{tikz}
\usetikzlibrary{patterns}
\usepackage{caption}
\usepackage{subcaption}

\usepackage{relsize} 
\newcommand{\euler}{\mathrm{e}}
\newcommand{\RR}{\mathbb{R}}

\newcommand{\NN}{\mathbb{N}}    
\newcommand{\ZZ}{\mathbb{Z}}

\newcommand{\TT}{\mathbb{T}}

\newcommand{\vol}{\operatorname{Vol}}
\newcommand{\cF}{\mathcal{F}}

\newcommand{\cD}{\mathcal{D}}

\newcommand{\drm}{\mathrm{d}}

\newcommand{\chiS}{\mathbf{1}_S}
\newcommand{\gap}{\operatorname{gap}}

\newcommand{\V}{\mathcal{V}}

\newtheorem{theorem}{Theorem}
\newtheorem{lemma}[theorem]{Lemma}
\newtheorem{proposition}[theorem]{Proposition}
\newtheorem{corollary}[theorem]{Corollary}

\theoremstyle{definition}
\newtheorem{definition}[theorem]{Definition}
\newtheorem{remark}[theorem]{Remark}
\usepackage{microtype}
\usepackage{ellipsis}
\usepackage{typearea}
\usepackage[textsize=tiny,shadow]{todonotes}
\begin{document}
\title[Controllability of Schr\"odinger equation]{Controllability of the Schr\"odinger equation on unbounded domains without geometric control condition}
 \author{Matthias T\"aufer}
 \address[MT]{FernUniversit\"at in Hagen, Fakult\"at f\"ur Mathematik und Informatik}
 \keywords{Controllability, Schr\"odinger equation, Observability Estimate, Floquet-Bloch theory}
 \subjclass[2010]{93Bxx, 35Q93}
 %
 %
 
 \begin{abstract}
 We prove controllability of the Schr\"odinger equation in $\RR^d$ in any time $T > 0$ with internal control supported on nonempty, periodic, open sets.
  This demonstrates in particular that controllability of the Schr\"odinger equation in full space holds for a strictly larger class of control supports than for the wave equation and suggests that the control theory of Schr\"odinger equation in full space might be closer to the diffusive nature of the heat equation than to the ballistic nature of the wave equation.
  Our results are based on a combination of Floquet-Bloch theory with Ingham-type estimates on lacunary Fourier series.
 \end{abstract}

\maketitle

\section{Controllability of the Schr\"odinger equation in $\RR^d$}

	Consider the Schr\"odinger equation in $\RR^d$ with internal control
\begin{equation}
	\label{eq:Schroedinger_controlled}
	\begin{cases}
	i \frac{\drm}{\drm t} u + \Delta u = \chiS f
	& \quad
	\text{in $\RR^d \times [0,T]$},\\
	u(0) = u_0
	& \quad
	\in L^2 (\RR^d)\\
	\end{cases}
\end{equation}
where $S \subset \RR^d$ is the \emph{control set}, 
$\chiS$ its indicator function, and $f \in L^2(S \times (0,T))$ is called the \emph{control}.
For any choice of $f$ and $u_0$, the solution of~\eqref{eq:Schroedinger_controlled} is given by the Duhamel formula
\begin{equation*}
	\label{eq:Duhamel}
	u(t)
	=
	\euler^{i t \Delta} u_0
	+
	\int_0^t \euler^{i (t - s) \Delta} \chiS f(s) \drm s
	.
\end{equation*}

\begin{definition}
	Equation~\eqref{eq:Schroedinger_controlled} is called \emph{controllable in time $T > 0$} if for all $u_0, u_T \in L^2(\RR^d)$ there exists $f \in L^2(S \times (0,T))$ such that the solution of~\eqref{eq:Duhamel} satisfies $u(T) = u_T$.  
\end{definition}

Our main result is:

\begin{theorem}
	\label{thm:control_periodic}
	Let $S \subset \RR^d$ be nonempty, periodic and open.
	Then~\eqref{eq:Schroedinger_controlled} is controllable in any time $T > 0$.
\end{theorem}

\begin{remark}
Since the free Schr\"odinger group $(\euler^{i t \Delta})_{t \in \RR}$ is unitary, controllability in time $T$ is equivalent to \emph{null-controllability in time $T$}, that means that for every $u_0 \in L^2(\RR^d)$ there is a control $f$ such that $u_T = 0$.
Indeed, a null-control for the initial state 
\[
\tilde u_0 := u_0 - \euler^{- i T \Delta} u_T
\]
will send the initial state $u_0$ to $u_T$ in time $T$.
Therefore, from now on, we will focus on null-controllability and assume that the target state $u_T = 0$.
\end{remark}

Let us comment on the implications of Theorem~\ref{thm:control_periodic}:
In recent years, control problems on unbounded domains have gained attention -- on the one hand motivated by applications to kinetic theory~\cite{DickeSV-22}, on the other hand in order to further the geometric understanding of problems on bounded domains by reconciling phenomena on bounded and unbounded domains and using unbounded domains as a proxy for large domains or a family of growing domains where the control set has a particular structure~\cite{TaeuferT-17,Taeufer-diss,NakicTTV-18,NakicTTV-20, WangWZZ-19}.
In parts this has coincided with progress in the understanding of  Anderson localization for random Schr\"odinger operators where conditions on the support of random potentials and the form of randomness causing localization have been succesively relaxed~\cite{TaeuferV-15,NakicTTV-15,TaeuferV-16,TaeuferT-17b, TaeuferV-21}
Thus, also the task of identifying all sets $S \subset \RR^d$ that ensure controllability of controlled systems has recently gained attention -- for the controlled Schr\"odinger equation as well as for its two cousins --
the \emph{controlled wave equation}
\begin{equation}
	\label{eq:Wave_controlled}
	\begin{cases}
	\frac{\drm^2}{\drm t^2} u - \Delta u = \chiS f
	& \quad
	\text{in $\RR^d \times [0,T]$},\\
	\left( (u(0), \frac{\partial}{\partial t} u(0) \right) = (u_0, u_1),
	&
	\\
	\end{cases}
\end{equation}
and the \emph{controlled heat equation}
\begin{equation}
	\label{eq:Heat_controlled}
	\begin{cases}
	\frac{\drm}{\drm t} u - \Delta u = \chiS f
	& \quad
	\text{in $\RR^d \times [0,T]$},\\
	u(0) = u_0.
	& \\
	\end{cases}
\end{equation}
For the latter two, all sets $S$ leading to null-controllability in time $T$ have recently been identified:

\begin{definition}
		A set $S \subset \RR^d$ satisfies the \emph{Geometric Control Condition} (GCC) if there are $\delta, L > 0$ such that every straight line of length $L$ has intersection at least $\delta$ with $S$ (with respect to the one-dimensional Hausdorff measure on the line).
		\\
		A set $S \subset \RR^d$ is called \emph{thick} if there are $\gamma, \rho > 0$ such that $\vol ( S \cap B_\rho(x)) \geq \gamma$ for all $x \in \RR^d$.		
\end{definition}

\begin{figure}[ht]

	\begin{tikzpicture}[scale =.75]
	\begin{scope}[xshift = -6cm]
	\draw[dotted] (-.2, -.2) grid (5.2,5.2);
	\foreach \x in {0.5,1.5,...,4.5}{
  	\foreach \y in {0.5,1.5,...,4.5}{
  	\pgfmathrandominteger{\radius}{1}{10}
    \filldraw[fill=gray!70] (\x,\y) circle (0.25);
  }
}
\draw (2.5,-.5) node {Nonempty, open, periodic.};

\end{scope}

	\begin{scope}[xshift = 0cm]
	\fill[gray!70] (0,0) rectangle (5,5);

	\draw[dotted] (-.2, -.2) grid (5.2,5.2);
	\foreach \x in {0.5,1.5,...,4.5}{
  	\foreach \y in {0.5,1.5,...,4.5}{
    \filldraw[fill=white] (\x,\y) circle (0.35);
  }
}
\draw (2.5,-.5) node {Set satisfying the GCC.};

	\end{scope}

\begin{scope}[xshift=6cm]
  \pgfmathsetmacro{\X}{4};
  \pgfmathsetmacro{\Y}{4};
  \pgfmathsetmacro{\e}{4};
  \pgfmathsetmacro{\d}{4};
  \foreach \y in {0,...,\Y}{
	\foreach \x in {0,...,\X}{
		\pgfmathsetmacro{\A}{random(1,\e)}; 
		\pgfmathsetmacro{\a}{\A^2-1};
                \pgfmathsetmacro{\B}{random(1,\d)};  
		\pgfmathsetmacro{\b}{\B^2-1};
		\pgfmathsetmacro{\Lo}{1/\A^2}; 
		\pgfmathsetmacro{\Lv}{1/\B^2}; 


		\foreach \i in {0,..., \a}{
			  \pgfmathsetmacro{\v}{\x +(\i *\Lo)};
			  \pgfmathsetmacro{\V}{\x +(\i +1)*\Lo};
			  \foreach \j in {0,...,\b}{
				\pgfmathsetmacro{\W}{\y + (\j +1)*\Lv};
				\pgfmathsetmacro{\w}{\y +(\j *\Lv)};
						\pgfmathsetmacro{\test}{\j +\i};
						\ifthenelse{\isodd{\test}}{
						}{\filldraw[black!40] (\v,\w) rectangle (\V, \W);}
			}
		}
		\draw (\x,\y) rectangle (\x+1, \y+1);
	}
  }
\draw (2.5,-.5) node {Generic thick set.};
\end{scope}
\end{tikzpicture}

\caption{Different types of control sets discussed in this article. Note that the nonempty, open, periodic arrangement of balls on the left does not satisfy the GCC since there are arbitrarily long straight lines in its complement.}
\label{fig:sets}

\end{figure}
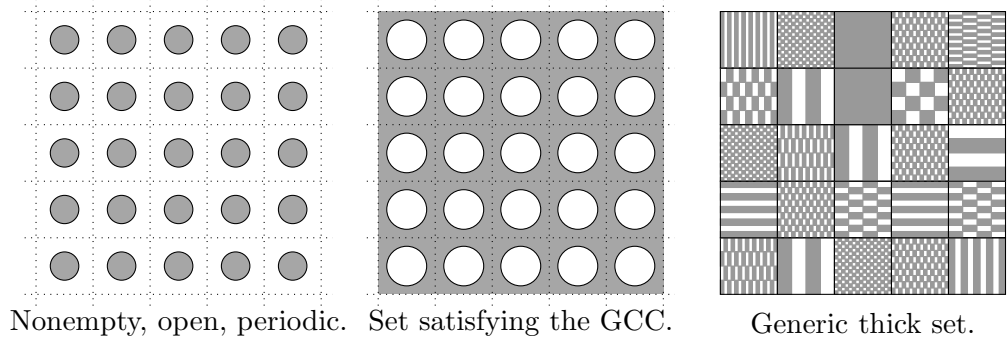

\begin{remark}
	The GCC is a strictly stronger condition that thickness. One example of thick sets not satisfying the GCC are periodic arrangements of balls with small radii as permitted in Theorem~\ref{thm:control_periodic}.
	
	In the literature, one often finds the GCC formulated in terms of geodesics for a Riemannian metric, i.e. tailored to controllability of differential operators of the form $\operatorname{div} A \nabla$ instead of the pure Laplacian, and there are versions for domains or manifolds with non-empty boundary.
	Since we are interested in the pure Laplacian in full space, we refrain from formulating these more general variants.
	Also note that the GCC reflects the ballistic nature of the wave equation whereas thickness reflects the diffusive nature of the heat equation.
\end{remark}

In 2016, Burq and Joly proved that the GCC is the necessary and sufficient criterion on the set $S \subset \RR^d$ for controllability of the wave equation. 

\begin{proposition}[{\cite[Section 6.2]{BurqJ-16}}]
	\label{prop:wave}
	The controlled wave equation in $\RR^d$ is controllable in some time $T$ if and only if $S \subset \RR^d$ satisfies the GCC.
\end{proposition}

Note that for the wave equation there will be a -- in general positive -- infimal time $T$ (depending on the parameters $L, \delta$ in the definition of the GCC) in which it is controllable -- a consequence of finite speed of propagation.

As for the heat equation in full space, if null-controllability holds in some time then it holds for all times and, in 2018 and 2019 two articles by Egidi and Veselic~\cite{EgidiV-18} as well as by Wang, Wang, Zhang and Zhang~\cite{WangWZZ-19} independently proved that the necessary and sufficient criterion is thickness.

\begin{proposition}[\cite{EgidiV-18, WangWZZ-19}]
	\label{prop:heat}
	The controlled heat equation in $\RR^d$~\eqref{eq:Heat_controlled} is null-controllable in any time $T > 0$ if and only if $S \subset \RR^d$ is thick.
\end{proposition}

Now, the Schr\"odinger equation shares properties both with the wave equation (time-reversability) and with the heat equation (infinite speed of propagation).
It is occasionally described as a chimera of the two or as a \emph{wave equation with infinite speed of propagation}. 
This raises the questions which criterion is necessary and sufficient for controllability of the Schr\"odinger equation.
So far, this has only been answered in dimension one, where the concept of thickness and the GCC happen to coincide. Indeed, the references~\cite{MartinPS-20, HuangWW-22} independently proved that they are the necessary and sufficient criterion for controllability of the one-dimensional Schr\"odinger equation.
 
In higher dimensions, there are only preliminary results for the controlled Schr\"odinger equation so far.
It is known that, if the control set is a complement of a ball, the Schr\"odinger equation is null-controllable~\cite{RosierZ-09}, with explicit estimates on the cost~\cite{WangWZ-19}.
Also, the following theorem stating that the GCC is sufficient for controllability in any dimension sees to be folklore:

\begin{theorem}
	If the control set satisfies the GCC, then the Schr\"odinger equation is null controllable in any time $T > 0$.
\end{theorem}

\begin{proof}
	If $S$ satisfies the GCC then the \emph{wave equation}~\eqref{eq:Wave_controlled} is controllable in some time $T^\ast > 0$ by Proposition~\ref{prop:wave}.
	But a result by Miller immediately turns this into controllability of the Schr\"odinger equation in any time $T > 0$.
	Indeed this is the statement of, \cite[Theorem~3.1]{Miller-05}, with $A = - \Delta$.
	\footnote{Strictly speaking, the application of~\cite[Theorem~3.1]{Miller-05} leads to controllability of the Schr\"odinger equation with a different sign convention $i \frac{\drm}{\drm t} u - \Delta u = \mathbf{1}_S f$, but this merely amounts to a complex conjugation.}
\end{proof}

Likewise, thickness is a known necessary condition:

\begin{theorem}[{Theorem~1.6 in~\cite{MartinPS-20} and Theorem~1.1 in~\cite{HuangWW-22}}]
	\label{thm:thickness_necessary} 
	Let the controlled Schr\"odinger equation be controllable in some time $T > 0$.
	Then $S$ must be thick.
\end{theorem}

For the sake of self-containedness we provide a proof in Section~\ref{sec:thickness_necessary}.
It differs from the argument in~\cite{MartinPS-20, HuangWW-22} which rely on works in~\cite{Miller-05} in combination with Kovrijkine's version of the Logvinenko-Sereda theorem~\cite{LogvinenkoS-74,Kovrijkine-00}. 
Instead, we use an elementary construction of explicit counterexamples to an observability inequality for non-thick sets.

\begin{figure}[ht]
	\begin{tikzpicture}
		\begin{scope}
			\draw[thick] (0,0) ellipse (4cm and 2cm);
			\draw[thick, fill = black!20] (0,-1.5) -- (0,1.5) arc (90:270:1.5);
			\draw[dashed, thick, pattern = north east lines, pattern color = black!60] (0,1.5) -- (0,-1.5) arc (-90:90:1.5);

			\draw[fill = white] (-.75,0) node {\footnotesize GCC};
			\draw (.75,.5) node {\footnotesize open,};
			\draw (.75,0) node {\footnotesize nonempty,};
			\draw (.75,-.5) node {\footnotesize periodic};
			
			\draw (-3,0) node {\footnotesize thick};
		\end{scope}
		
		\begin{scope}[xshift = 4.5cm]
			\draw[thick, fill = black!20] (0,-1.25) rectangle (1,-.75);

			\draw[white, fill = black!20] (0,-.25) rectangle (.5,.25);
			\draw[pattern = north east lines, pattern color = black!60] (.5,-.25) rectangle (1,.25);
			\draw[thick] (0,-.25) rectangle (1,.25);		
			\draw[thick] (0,.75) rectangle (1, 1.25);
			
			\draw[anchor = west] (1,1) node {\footnotesize heat equation controllable};
			\draw[anchor = west] (1,0) node {\footnotesize Schr\"odinger equation controllable};
			\draw[anchor = west] (1,-1) node {\footnotesize wave equation controllable};
		\end{scope}
	\end{tikzpicture}
	\caption{Known necessary and sufficient conditions on the control sets $S \subset \RR^d$ for the heat, Schr\"odinger and wave equation.}

\end{figure}
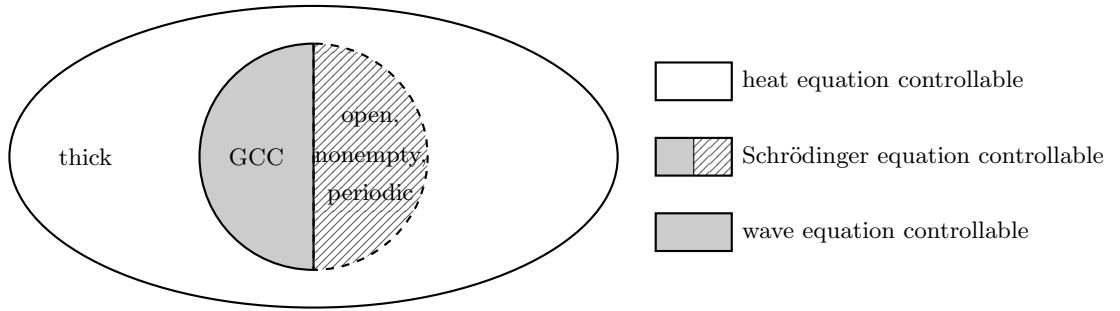

Now, in order to return to the interpretation of Theorem~\ref{thm:control_periodic}, we can summarize our findings in the following 
\begin{corollary}
	The set of all $S \subset \RR^d$ such that the controlled Schr\"odinger equation in $\RR^d$ with control in $S$ is controllable in any time $T > 0$ is contained in the corresponding set for the heat equation (thick sets), but \emph{strictly larger} than the corresponding set for the wave equation (sets satisfying the GCC).
\end{corollary}

One consequence is that when designing control systems for the Schr\"odinger equation, one has a larger choice for the control support $S \subset \RR^d$ than for the wave equation.

It would be an interesting question to investigate whether the GCC still yields any benefit for controllability of the Schr\"odinger equation -- e.g. on the level of the control cost.
Likewise, the question of identifying the class of \emph{all} sets $S \subset \RR^d$ which lead to controllability of the Schr\"odinger equationin full space remains open.

	We also expect our results not necessarily be restricted to full space.
	Indeed, by reflection techniques, see for instance~\cite{EgidiS-22} for analogous arguments in the case of the heat equation, Theorem~\ref{thm:control_periodic} will also holds for half-spaces, sectors or certain infinite cones.
	Also, infinite strips or slabs can be treated by a modification of the proof of Theorem~\ref{thm:control_periodic} in Section~\ref{sec:Floquet}, taking the Floquet transform only in those directions in which the domain is unbounded.

\section{Proofs}

	Without loss of generality, we may assume that the periodic set $S$ is a union of balls of radius $R > 0$, centered at the lattice points of $2 \pi \ZZ^d$
	\[
	S 
	=
	\bigcup_{k \in 2 \pi \ZZ^d}
	B_R(k).
	\]
	By the Hilbert Uniqueness method \cite{Lions-88}, (null)-controllability in time $T$ is equivalent to the observability inequality
\begin{equation}
	\label{eq:observability}
 	\lVert u \rVert_{L^2(\RR^d)}^2
  	\leq
  	C
  	\int_0^T
  	\lVert \chiS \euler^{i t \Delta} u \rVert_{L^2(\RR^d)}^2 \drm t
  	\quad
  	\text{for all}
  	\quad
  	u \in L^2(\RR^d).
\end{equation}

\subsection{Floquet-Bloch decomposition}
	\label{sec:Floquet}
	We use Floquet-Bloch theory to reduce the observability estimate~\eqref{eq:observability} to a family of observability estimates on compact domains with varying boundary conditions.  
	Let $\TT^d :=  \RR^d / (2 \pi \ZZ^d) \cong (-\pi, \pi]^d$ denote the $d$-dimensional torus.
		
	\begin{definition}
 	Denote by $\cF \colon L^2(\RR^d) \to L^2(\TT^d \times \TT^d)$ the \emph{Floquet transform}
		\begin{align*}
		(\cF u) (y, \theta)
		&=
		{(2 \pi)^{-d/2}}
		\sum_{k \in 2 \pi \ZZ^d} 
		\euler^{i \theta k}
		u(y + k). 
		\end{align*}
	\end{definition}	
	This is an isometric isomorphism that commutes with periodic functions in the sense that if $u \in L^2(\RR^d)$ and $f \in L^\infty(\RR^d)$ is $2 \pi \ZZ^d$-periodic, then
		\[
		\cF (f \cdot u) (y, \theta)
		=
		f \mid_{\TT^d} (y) \cdot (\cF u)  (y, \theta).
		\]
	One also has for $u \in H^2(\RR^d)$
	\[
	\cF (\Delta u) (y, \theta)
	=
	\Delta_\theta
	(\cF u) (y, \theta)
	\]
	for almost all $\theta \in \TT^d$ where $\Delta_\theta$ is the \emph{self-adjoint Laplacian on $\TT^d$ (in the $y$-coordinate) with $\theta$-pseudoperiodic boundary conditions}, that means
	\[
	\cD(\Delta_\theta) 
	= 
	\left\{
	\phi \in L^2(\TT^d) \colon
	\begin{aligned}
	&\text{there exists\ }	
	\tilde \phi \in H^2_{\mathrm{loc}}(\RR^d)\ \text{with}\
	\tilde \phi \mid_{\TT^d} = \phi,\ \text{and}
	\\
	&
	\tilde \phi( x + 2 \pi k) = \exp(i k \cdot \theta) \tilde \phi(x)\
	\text{for all $k \in \ZZ^d$, $x \in \RR^d$}
	\end{aligned}
	\right\}.
	\]
	Indeed, these statements can be seen by straightforward calculations on sufficiently smooth functions and the general case follows by density.
	We also refer to~\cite{Sjoestrand-91, Kuchment-93} for further reading on Floquet theory for partial differential operators.
	In particular, one also has for all $u \in L^2(\RR^d)$
	\[
	\cF (\euler^{i t \Delta} u) (y, \theta)
	=
	\euler^{i t \Delta_\theta}
	(\cF u) (y, \theta).
	\]
	
	The spectral decomposition of the operator $\Delta_\theta$ is well-known:
	\begin{lemma}
		\label{lem:eigenbasis}
		The spectral decomposition of $\Delta_\theta$ is given by the eigenbasis
		\[	
		(\euler^{i \gamma y})_{\gamma \in \Gamma_\theta},
		\text{where}\
		\Gamma_\theta
		:=
		\frac{\theta}{2 \pi} + \ZZ^d
		\]
		with corresponding eigenvalues
		\[
		(\lvert \gamma \rvert^2)_{\gamma \in \Gamma_\theta}
		.
		\]
	\end{lemma}
\begin{proof}
Since $(\euler^{i k y})_{k \in \ZZ^d}$ is an orthogonal basis of $L^2(\TT^d)$, so is $(\euler^{i \gamma y})_{\gamma \in \Gamma_\theta}$.
The eigenvalue property and the $\theta$-pseudoperiodicity are verified by explicit computation.
\end{proof}
	By isometry of the the Floquet transform, the l.h.s. of~\eqref{eq:observability} becomes
	\begin{align*}
	\int_0^T
	\int_S
	\lvert 
	(\euler^{i t \Delta} u) (x)
	\rvert^2
	\drm x\
	\drm t
	&=
	\int_0^T
	\drm t
	\left\langle
	\euler^{i t \Delta} u
	,
	\chiS
	\euler^{i t \Delta} u
	\right\rangle_{\RR^d}
	\\
	&=
	\int_0^T
	\drm t
	\left\langle
	\cF (\euler^{i t \Delta} u)
	,
	\cF
	(\chiS
	\euler^{i t \Delta} u)
	\right\rangle_{\TT^d \times \TT^d}
	\\
	&=
	\int_{\TT^d}
	\drm \theta
	\int_0^T
	\drm t
	\left\langle
	\euler^{i t \Delta_\theta} (\cF u)(\cdot,\theta)
	,
	\mathbf{1}_{B_R}
	\euler^{i t \Delta_\theta} (\cF u)(\cdot,\theta)
	\right\rangle_{\TT^d}
	\\
	&=
	\int_{\TT^d}
	\drm \theta
	\int_0^T
	\drm t
	\int_{B_R}
	\drm y
	\lvert
	\euler^{i t \Delta_\theta} (\cF u)(y,\theta)
	\rvert^2
	\\
	&=
	\int_{\TT^d}
	\drm \theta
	\int_0^T
	\drm t
	\int_{B_R}
	\drm y
	\big\lvert
	\sum_{\gamma \in \Gamma_\theta}
	\euler^{i (t \lvert \gamma \rvert^2 + y \gamma)}
	\alpha_{\theta, \tilde \gamma}
	\big\rvert^2
	\end{align*}
	where the $\alpha_\theta = (\alpha_{\theta, \tilde \gamma})_{\gamma \in \Gamma_\theta} \in \ell^2(\tilde \Gamma_\theta)$ are the Fourier coefficients of 
	$\TT^d \ni y \mapsto \cF (u) (y, \theta) \in L^2(\TT^d)$
	with respect to the eigenbasis from Lemma~\ref{lem:eigenbasis} for fixed $\theta \in \TT^d$.
	Writing
	\[
	t \lvert \gamma \rvert^2 + y \gamma 
	=
	\left\langle
	\begin{pmatrix}
	y\\
	t\\
	\end{pmatrix}
	,
	\begin{pmatrix}
	\gamma \\
	\lvert \gamma \rvert^2\\
	\end{pmatrix}
	\right\rangle
	=:
	z \cdot \tilde \gamma,
	\]
	for $\tilde \gamma := (\gamma, \lvert \gamma \rvert^2) \in \RR^{d+1}$, using unitarity of the Schr\"odinger group to shift the integration in $t$ from $[0,T]$ to $[-\frac{T}{2}, \frac{T}{2}]$ and using that, possibly after shrinking R, $[-\frac{T}{2}, \frac{T}{2}] \times B_R$ certainly contains the $(d+1)$-dimensional ball of radius $R$ around the origin (which we again denote by $B_R$), we can lower bound this by
	\[
	\int_{B_R}
	\big\lvert
	\sum_{\tilde \gamma \in \tilde \Gamma_\theta}
	\exp \left( z \cdot \tilde \gamma \right)
	\alpha_{\theta, \tilde \gamma}
	\big\rvert^2
	\drm z.
	\]
	We have thus reduced the proof of the observability estimate~\eqref{eq:observability} to the following statement:
	There exists $C > 0$ such that for all $\theta \in \TT^d$ one has
	\begin{equation}
	\label{eq:exponential_sum}
	\int_{B_R}
	\big\lvert
	\sum_{\tilde \gamma \in \tilde \Gamma_\theta}
	\exp \left( z \cdot \tilde \gamma \right)
	\alpha_{\theta, \tilde \gamma}
	\big\rvert^2
	\drm z
	\geq
	C
	\lVert \alpha_\theta \rVert_{\ell^2(\tilde \Gamma_\theta)}^2
	\quad
	\text{for all $\alpha_\theta \in \ell^2(\tilde \Gamma_\theta)$},
	\end{equation}
	where
	\[
	\tilde \Gamma_\theta 
	:=
	\left\{
		\begin{pmatrix}
			\gamma \\ \lvert \gamma \rvert^2
		\end{pmatrix}
		\colon
		\gamma \in \frac{\theta}{2 \pi} + \ZZ^d
	\right\}
	\subset \RR^{d+1}.
	\]
	
%
	
	\subsection{Lower bounds on exponential sums}
	
	We proceed to prove~\eqref{eq:exponential_sum}. 
	For $\theta = 0$, this can be found in the $2$-dimensional case in~\cite{Jaffard-90} as an ingredient in a proof of controllability of the Schr\"odinger equation on rectangles, and for higher dimensions in~\cite{KomornikL-05}.
	While the proofs therein immediately generalize to \emph{fixed} $\theta$, we need to ensure that the resulting constants are $\theta$-independent -- a statement not explicitly provided in~\cite{KomornikL-05}.
	We therefore follow the central parts of the proofs given therein.
	
	For a discrete set $X \subset \RR^{d+1}$, let its \emph{gap} be given by
	\[
	\gap(X) 
	:=
	\inf
	\{
	\lvert x - y \rvert
	\colon
	x, y \in X, x \neq y
	\}.
	\]	
	The following lemma can be inferred in this form from Theorem~8.1 in~\cite{KomornikL-05}, see also~\cite[Proposition~1]{Jaffard-90}.
	In one dimension, the lemma is due to Ingham~\cite{Ingham-36}, while the higher dimensional case is essentially due to Kahane~\cite{Kahane-62}, see also~\cite{BaiocchiKL-99}.

	\begin{lemma}
	\label{lem:Ingham}
	There is $c > 0$, depending only on the dimension such that for all $\delta > 0$ there is $\tilde C > 0$, depending only on the dimension and on $\delta$, such that for all $R \geq \frac{c}{\delta}$ and all $\Gamma \subset \RR^{d+1}$ with $\gap(\Gamma) \geq \delta$ we have
	\[
	\int_{B_R} \lvert \sum_{\gamma \in \Gamma} \alpha_\gamma \exp ( i \gamma x ) \rvert^2 \drm x 
	\geq
	\tilde C
	\sum_{\gamma \in \Gamma} \lvert \alpha_\gamma \rvert^2
	\quad
	\text{for all $\alpha \in \ell^2( \Gamma)$}.
	\] 
	\end{lemma}
	
\begin{remark}
	\label{rem:Bessel}
	There is an explicit expression for the infimum of all possible $c$, see the remark after~\cite[Theorem~8.1]{KomornikL-05}.
	More precisely, in dimension $d \geq 2$, the infimal $c$ is the first positive root of $J_{\frac{d-1}{2}}$, the Bessel function of first kind with parameter $\frac{d-1}{2}$.
	Choosing a strictly larger $c$, one obtains a uniform $\tilde C$ for all $R \geq \frac{c}{\delta}$.
\end{remark}

	By demanding $R \geq \frac{c}{\delta}$, Lemma~\ref{lem:Ingham} imposes a reciprocal relation between the radius $R$ of the ball and the gap $\delta$.
	Since the gap of $\Gamma_\theta$ is of order one, this would exclude small $R$.
	The following lemma remedies this by allowing to decompose a discrete $\Gamma$ into finitely many $\Gamma_j$ with larger gaps.

	\begin{lemma}
	\label{lem:several_sets}
	Let $R > 0$, $N \in \NN$, and $c > 0$ be the constant from Lemma~\ref{lem:Ingham}.
	Let $\Gamma = \Gamma_1 \cup \dots \cup \Gamma_N$ be such that $\gap(\Gamma) \geq 1$, and $\gap(\Gamma_i) \geq \delta_i$.
	Then, there is $C > 0$, depending only on the dimension, on $N$, and on the $\delta_i$ such that for all 
	\begin{equation}
	\label{eq:prefactor_2}
	R 
	\geq
	2
	\left( \frac{c}{\delta_1} + \dots + \frac{c}{\delta_N} \right)
	\end{equation}
	 we have
	\[
	\int_{B_R} \lvert \sum_{\gamma \in \Gamma} \alpha_\gamma \exp ( i \gamma x ) \rvert^2 \drm x 
	\geq
	C
	\sum_{\gamma \in \Gamma} \lvert \alpha_\gamma \rvert^2
	\quad
	\text{for all $\alpha \in \ell^2(\tilde \Gamma)$}.
	\] 
	\end{lemma}
	
	Here, we use the convention that for a one-element set $\gap ( \Gamma ) = \infty$ and  $\frac{1}{\infty} = 0$.
		A qualitative variant of Lemma~\ref{lem:several_sets} is due to~\cite{Kahane-62}, whereas we refer here again to~\cite[Proposition~8.4]{KomornikL-05} from which the quantitative version stated above can be inferred in combination with Lemma~\ref{lem:Ingham}.
	While the statement a priori holds for all
	\[
	R 
	>
	\left( \frac{c}{\delta_1} + \dots + \frac{c}{\delta_N} \right),	
	\]
	the resulting constant $C$ will depend on a $\epsilon > 0$ with 
	\[
	R = 
	\left( 
		\frac{c}{\delta_1} + \dots + \frac{c}{\delta_N} 
	\right)
	+
	\epsilon,
	\]
	chosen in the proof in~\cite{KomornikL-05}.
	The prefactor $2$ in~\eqref{eq:prefactor_2} thus allows for a uniform choice, depending only on the dimension, on $N$, and on the $\delta_i$.

Now, the next lemma provides an appropriate decomposition of the set $\Gamma_\theta$
into appropriate subsets.
For $p = 2$, the $\theta$-independent version is in \cite{Jaffard-90} in dimension $d = 2$ and in \cite{KomornikL-05} in the higher-dimensional case.
A $\theta$-independent version for $p > 1$ can also be obtained as a corollary of~\cite[{Th\'eor\`eme~III.4.1}]{Kahane-62}.
However, for the required $\theta$-independence of the constant $C$ in Lemma~\ref{lem:several_sets}, we provide a proof here.

\begin{lemma}
	\label{lem:decomposition}
	Let $R > 0$, and $p > 1$.
	There is $N \in \NN$ such that for all $\theta \in \TT^d$, we can decompose
	\[
	\tilde \Gamma_\theta 
	=
	\left\{
		\begin{pmatrix}
			\gamma \\ \lvert \gamma \rvert^2
		\end{pmatrix}
		\colon
		\gamma \in \frac{\theta}{2 \pi} + \ZZ^d
	\right\} 
	=
	\bigcup_{j = 1}^N
	\tilde \Gamma_{\theta, j} 
	\]
	where for $\delta_j := \gap(\tilde \Gamma_{\theta, j})$ one has
	\[
	2
	\left( \frac{c}{\delta_1} + \dots + \frac{c}{\delta_N} \right)
	\leq
	R
	\]
	\end{lemma}

\begin{proof}
	We split $\tilde \Gamma_\theta$ into three components.
	
	\emph{1. Points with $(x_1, \dots, x_d)$ bounded away from the axes.}
	Defining for $\alpha \geq 1$ 
	\[
	A_\alpha
	:= 
	\left\{
	x \in \RR^{d+1} 
	\colon
	\lvert x_j \rvert \geq \alpha
	\
	\forall j \in \{1,\dots d\}
	\right\}
	\]
	one has
	\[
	\gap(\tilde \Gamma_\theta \cap A_\alpha) 
	\geq 
	\alpha.
	\]
	Indeed, let $\gamma \in \Gamma_\theta$ with $( \gamma, \lvert \gamma \rvert^p) \in A_\alpha$ and $e_k = (0, \dots, 0,\pm 1,0, \dots, 0) \in\ZZ^d$ be the $k$-th unit vector.
	Then, 
	\begin{align*}
	\left\lvert 		
		\begin{pmatrix}
			\gamma \\ \lvert \gamma \rvert^2
		\end{pmatrix}
	-
		\begin{pmatrix}
			\gamma + e_k \\ \lvert \gamma + e_k \rvert^2
		\end{pmatrix}
	\right\rvert^2
	&\geq
	\left(
	\lvert \gamma \rvert^2 - \lvert \gamma + e_k \rvert^2 
	\right)^2.
	=
	\left( \lvert \gamma \rvert^2 - \lvert \gamma \pm e_k \rvert^2 \right)^2
	\\
    &
	\geq
	\left(
		\gamma_k^2
		-
		(\gamma_k \pm 1) ^2
	\right)^2
	\geq
	(\pm 2 \gamma_k + 1)^2
	\geq
	(\alpha - 1)^2
	\geq
	\alpha^2.
	\end{align*}
	Thus, choosing
	\[
	\alpha 
	= 
		\max \left\{ \frac{6 c}{R}, 1 \right\}
	\] 
	and setting $\tilde \Gamma_{\theta,1} = A_\alpha \cap \tilde \Gamma_\theta$ ensures $\frac{2 c}{\delta_1} \leq \frac{R}{3}$.
	
	\emph{2. Neighbourhoods of the axes in the first $d$ coordinates}.
	Next, for $\beta > \alpha$ let 
	\[
	B_{\alpha, \beta} 
	:= 
	A_\alpha^c \cap \{ x \in \RR^{d+1} \colon \exists j \in \{1, \dots, d \} \colon \lvert x_j \rvert \geq \beta \}.
	\]
	The set $B_{\alpha, \beta}$ is a union of $2 d$ many slabs, consisting of rectangular neighbourhoods around positive the $x_1$ to $x_d$ axes in $\RR^d$, and extending infinitely in the $x_{d+1}$ direction.
	\begin{figure}[ht]
		\begin{tikzpicture}[scale =.5]
			
			\draw (7,5) node {$\RR^d$};
			
						\draw[fill = black!20] (-4.5,-2.5) -- (-2.5,-2.5) -- (-2.5,-4.5) -- (2.5,-4.5) -- (2.5,-2.5) -- (4.5,-2.5) -- (4.5,2.5) -- (2.5,2.5) -- (2.5,4.5) -- (-2.5,4.5) -- (-2.5,2.5) -- (-4.5,2.5);

			\draw[->] (-6,0) -- (6.5,0);
			\draw[->] (0,-6) -- (0,6.5);
			\foreach \x in {-6,...,5}
				{
				\foreach \y in {-6,...,5}
					{\draw[] (\x+.2,\y+.3) circle (2pt);}
				}
			
			\draw[thick] (-6,-2.5) -- (-2.5,-2.5) -- (-2.5,-6);
			\draw[thick] (-6,2.5) -- (-2.5,2.5) -- (-2.5,6);
			\draw[thick] (2.5,-6) -- (2.5,-2.5) -- (6,-2.5);
			\draw[thick] (2.5,6) -- (2.5,2.5) -- (6,2.5);
			
			\draw[thick] (-4.5,-2.5) -- (-4.5,2.5);
			\draw[thick] (4.5,-2.5) -- (4.5,2.5);
			\draw[thick] (-2.5,4.5) -- (2.5,4.5);
			\draw[thick] (-2.5,-4.5) -- (2.5,-4.5);

			\fill[pattern = north east lines] (-2.5,4.5) rectangle (2.5,6);
			\fill[pattern = north east lines] (-2.5,-4.5) rectangle (2.5,-6);
			\fill[pattern = north east lines] (-6,-2.5) rectangle (-4.5,2.5);
			\fill[pattern = north east lines] (4.5,-2.5) rectangle (6,2.5);

			\draw (2.5,-.1) -- (2.5,.1);
			\draw (2.5,-.4) node {\tiny$\alpha$};
			\draw (4.5,-.1) -- (4.5,.1);
			\draw (4.3,-.4) node {\tiny$\beta$};

			\begin{scope}[xshift = 7cm]
			\draw[thick] (1,1) rectangle (3,2);
			\draw[anchor = west] (4,1.5) node {$A_\alpha$};
			\draw[thick, pattern = north east lines] (1,-.5) rectangle (3,.5);
			\draw[anchor = west] (4,0) node {$B_{\alpha, \beta}$};			
			\draw[thick, fill = black!20] (1,-2) rectangle (3,-1);
			\draw[anchor = west] (4,-1.5) node {Rest};

			\end{scope}
		\end{tikzpicture}

		\caption{Projections onto $\RR^d$ of the decomposition of $\tilde \Gamma_\theta$.}
		\label{fig:decomposition}	
	\end{figure}
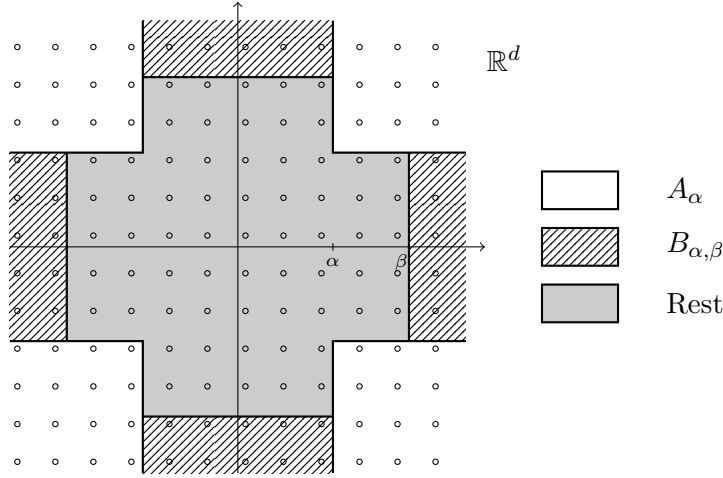
	Let $B_{\alpha, \beta,k}^{\pm} \subset B_{\alpha, \beta}$ denote the union of the two components which are arranged along the positive and negative part of the $x_k$-axis.
	Then, 
	\[
	\tilde \Gamma_\theta \cap B_{\alpha, \beta,k}^{\pm}
	=
	\left\{
		\begin{pmatrix}
			\gamma \\ \lvert \gamma \rvert^2
		\end{pmatrix}
		\colon
		\begin{aligned}
		&\gamma \in \ZZ^d + \theta,\
		\\
		&\lvert \gamma_k + \frac{\theta_k}{2 \pi} \rvert \geq \beta,
		\\
		&\lvert \gamma_i + \frac{\theta_i}{2 \pi} \rvert \leq \alpha 
		\text{ for all } i \neq k 
		\end{aligned}
	\right\}
	\]
	and we see that for sufficiently large $\beta$, the distance to the next neighbour in $x_k$ direction grows quadratically.
	Since $B_{\alpha, \beta,k}^{\pm}$ has a finite profile in $(x_1, \dots x_{k-1}, x_{k+1}, \dots, x_d)$ direction, we can split $B_{\alpha, \beta,k}^{\pm} \cap \tilde \Gamma_\theta$ into 
	\[
	N_{\alpha}
	:=
	\# 
	\{ 
	[- \alpha - 1, \alpha + 1]^{d-1}  
	\cap
    \ZZ^{d-1} + (\theta_1, \dots, \theta_{k-1}, \theta_{k+1}, \dots, \theta_d)
	\}
	\]
	many components, each with fixed $(x_1, \dots, x_{k-1}, x_{k+1}, \dots, x_d)$ coordinates, and with gap at least
	\[
	\lvert \beta + 1 \rvert^2 - \lvert \beta \rvert^2
	=
	2 \beta + \beta^2
	\geq
	3 \beta.
	\]
	Uniting all $d$ many components $\{B_{\alpha, \beta, k}^{\pm}\}_{k = 1}^d$ and choosing 
	\[
	\beta 
	\geq
		\frac{2d N_\alpha c}{R}
	\]
	we ensure
	\[
	2
	\sum_{j = 2}^{d N_\alpha + 1}
	\frac{c}{\delta_j}
	\leq
	\frac{R}{3}.
	\]
	
	\emph{3. The bounded rest}. 
	We have covered all but at most $\left( 2 \beta + 1 \right)^d$ many points of $\tilde \Gamma_\theta$.
	Let each of these points form another $\tilde \Gamma_{\theta, j}$ which has gap equal to $\infty$.
	We have found for each $\theta \in \TT^d$ a decomposition $\tilde \Gamma_\theta = \cup_{j = 1}^N \tilde \Gamma_{\theta, j}$ where $N$ is uniformly bounded in $\theta$ and~\eqref{eq:prefactor_2} holds.
\end{proof}

\begin{proof}[Proof of Theorem~\ref{thm:control_periodic}]
Combining Lemma~\ref{lem:decomposition} with Lemma~\ref{lem:several_sets} yields the $\theta$-independent observability estimate~\eqref{eq:exponential_sum}.
Due to the reductions made at the beginning of this section this proves the observability inequality~\eqref{eq:observability}, and consequently Theorem~\ref{thm:control_periodic}.
\end{proof}

\section{Thickness is necessary for Schr\"odinger}

\label{sec:thickness_necessary}

We give an elementary proof of Theorem~\ref{thm:thickness_necessary} (previously demonstrated using a different method in~\cite{MartinPS-20}) stating that in the case of controllability of the Schr\"odinger equation the control set $S$ must be thick.

\begin{proof}
	By the Hilbert uniqueness method~\cite{Lions-88}, (null-)controllability in time $T$ is equivalent to the observability inequality
 \[
  \lVert u \rVert^2
  \leq
  C
  \int_0^T
  \lVert \mathbf{1}_S \euler^{i t \Delta} u \rVert^2 \drm t
  \quad
  \text{for all}
  \quad
   u \in L^2(\RR^d).
 \]
 Thus, it suffices to find for any $\epsilon > 0$ a function $u \in L^2(\RR^d)$ with 
\[
 \int_0^T
 \lVert \mathbf{1}_S \euler^{i \Delta t} u \rVert^2 
 \drm t
 \leq
 \epsilon
 \lVert u \rVert^2.
\]
Clearly, we have for any $u \in L^2(\RR^d)$
\[
 \int_0^T
 \lVert \mathbf{1}_S \euler^{i \Delta t} u \rVert^2 
 \drm t 
 \leq
 2
 \int_0^T
 \lVert \mathbf{1}_S 
    \left[ 
        \euler^{i \Delta t} - \euler^{\Delta t} 
    \right] 
    u \rVert^2 
 \drm t
 +
 2
 \int_0^T
 \lVert \mathbf{1}_S \euler^{\Delta t} u \rVert^2 
 \drm t
 =
 (A) + (B).
\]
For $E > 0$ to be chosen below let $u = u_1 + u_2$ with $u_1$ having Fourier support in $B_E(0)$ and $u_2$ Fourier support in $B_E^c(0)$.
Since $0 \leq \mathbf{1}_S \leq \operatorname{Id}$ and since the operator norm of $\mathbf{1}_S \left[ \euler^{i \Delta t} - \euler^{\Delta t} \right]$ is at most $2$, we can estimate
\[
 (A)
 \leq
  2
 \int_0^T
 \lVert 
    \left[ 
        \euler^{i \Delta t} - \euler^{\Delta t} 
    \right] 
    u_1 \rVert^2 
 \drm t
 +
 4
 T
 \lVert u_2 \rVert^2 
 =
 (A_1) + (A_2)
\]
By Plancherel, we further estimate
\[
 (A_1) 
 =
 2
 \int_0^T
 \lVert 
 \left[ \euler^{- i \xi^2 t} - \euler^{-\xi^2 t} \right]
 \widehat{u_1}
 \rVert^2
 \drm t
 \leq
 2 T
 \max_{\lvert \xi \rvert \leq E, t \in [0,T]}
 \lvert \euler^{- i \xi^2 t} - \euler^{-\xi^2 t} \rvert
 \cdot
 \lVert u \rVert^2
 \leq
 c T^2 E^2 
 \lVert u \rVert^2.
\]
Taking $E$ sufficiently small (depending only on $T$) this can be made smaller than $\frac{\epsilon}{3} \lVert u \rVert^2$, independently of $u$.

We now choose $u$ as a (square root of a) normalized Gaussian with standard deviation $\nu$, and center $x_0$ that is
\[
 u(x)
 :=
 \frac{1}{\left(2 \pi \nu\right)^{d/4}}   
 \exp
 \left(
   - \frac{\lVert x - x_0 \rVert^2}{4 \nu}
 \right).
\]
where $\nu > 0$, $x_0 \in \RR^d$ will be chosen below.
We have
\[
 \widehat u(\xi)
 =
 \left( \frac{2 \nu}{\pi} \right)^{d/4}
 \exp 
 \left(
   - \xi^2 \nu
 \right),
 \quad
 \text{and}
 \quad
 \widehat u_2(\xi)
 =
 \chi_{\lvert \xi \rvert \geq E}
 \left( \frac{2 \nu}{\pi} \right)^{d/4}
 \exp 
 \left(
   - \xi^2 \nu
 \right).
\]
The function $\hat u$ is the (square root of a) Gaussian with standard deviation proportional to $\nu^{-1}$ and for $\nu \gg 1$, its mass will concentrate on $B_E(0)$. 
Hence, choosing $\nu$ sufficiently large (depending on $E$ which is already fixed), we obtain $(A_2) = 4 T \lVert \widehat u_2 \rVert^2 \leq \frac{\epsilon}{3} \lVert u \rVert^2$.

It remains to estimate $(B)$.
This is the time evolution part of the heat equation and can be made small by putting $x_0$ into a center of a sufficiently large (depending on $\epsilon$ and $\nu$) ball on which $\omega$ has little mass, as has been done for the heat equation, see for instance~\cite{EgidiV-18} for details.
Such a clearing with center $x_0$ exists by assumption since $S$ is assumed not to be thick.
\end{proof}

\section*{Acknowledgements}

The author thanks Kotaro Inami and Soichiro Suzuki for pointing out an error in an earlier version, Walton Green for discussions, and the anonymus referees for helpful comments.


\end{document}